\documentclass{amsart}

\usepackage{amsfonts,amssymb,amsmath,amsthm,latexsym,booktabs,lscape,color,enumerate, amscd}
\usepackage{marginnote}

\usepackage{calrsfs}
\DeclareMathAlphabet{\pazocal}{OMS}{zplm}{m}{n}

\usepackage{tikz}

\theoremstyle{definition}
\newtheorem{definition}{Definition}[section]
\newtheorem{remark}[definition]{Remark}
\newtheorem{example}[definition]{Example}

\theoremstyle{plain}
\newtheorem{lemma}[definition]{Lemma}
\newtheorem{proposition}[definition]{Proposition}
\newtheorem{theorem}[definition]{Theorem}
\newtheorem{corollary}[definition]{Corollary}

\newtheorem{conclu}[definition]{Conclusion}

\def\esc#1{\langle #1\rangle}
\def\l{\lambda}
\def\A{{\mathcal A}_K}
\def\B{{\mathcal B}_K}

\def\R{\mathbb R}
\def\C{\mathbb C}

\usepackage{color}

\newcommand{\Ann}{\mathrm{Ann}}

\tikzset{every picture/.style={line width=0.11mm}}

\makeatletter
\@namedef{subjclassname@2020}{%
  \textup{2020} Mathematics Subject Classification}
\makeatother



\begin{document}

\title[Evolution algebras with one-dimensional square]{Evolution algebras with one-dimensional square}

\author[Brache, Mart\'in Barquero, Mart\'in Gonz\'alez, S\'anchez-Ortega]
{Chad Brache$^{3}$, Dolores Mart\'in Barquero$^{1}$, C\'andido Mart\'in Gonz\'alez$^{2}$, Juana S\'anchez-Ortega$^{3}$}
\address[1]{Departamento de Matem\'atica Aplicada. Escuela de Ingenier\'\i as Industriales. Universidad de M\'alaga, Campus de Teatinos. 29071 M\'alaga,   Spain.}
\address[2]{Departamento de \'Algebra, Geometr\'ia y Topolog\'ia. Facultad de Ciencias. Universidad de M\'alaga, Campus de Teatinos. 29071 M\'alaga, Spain.}
\address[3]{Department of Mathematics and Applied Mathematics. University of Cape Town \\ Cape Town, South Africa.}
\email{dmartin@uma.es}
\email{candido@apncs.cie.uma.es}
\email{BRCCHA005@myuct.ac.za}
\email{juana.sanchez-ortega@uct.ac.za}

\thanks{ The first  and third authors are supported by the Spanish Ministerio de Ciencia e Innovaci\'on through project  PID2019-104236GB-I00 and  by the Junta de Andaluc\'{\i}a  through projects  FQM-336 and UMA18-FEDERJA-119,  all of them with FEDER funds. The second and fourth authors are supported by the National Research Foundation (NRF) South Africa through the grants no. MND200528525542 and IFR180306315977, respectively. The fourth author is also supported by the Faculty of Science of the University of South Africa via the Launching Grants Programme.
}

\maketitle


\begin{abstract}

Evolution algebras with one dimensional square are classified using the theory of inner product spaces. More precisely, for $A$ an evolution algebra with $\dim(A^2) = 1$ and $a$ a generator of $A^2$, the product of $A$ is given by $xy = \esc{x,y}a$. Three broad classes of algebras are obtained: 
\begin{enumerate}
\item $a\in\Ann(A)$; 
\item $a\notin\Ann(A)$ and $a$ is isotropic relative to $\esc{\cdot, \cdot}$; 
\item $a\notin\Ann(A)$ and $a$ is anisotropic relative to $\esc{\cdot, \cdot}$.
\end{enumerate}
\end{abstract}


\section{Introduction}

Evolution algebras were introduced by Tian in his book \cite{T1} (see also \cite{TV}) to model the self reproduction process in non-Mendelian genetics. As shown in \cite{T1}, the theory of evolution algebras is connected to many areas of Mathematics, like, for example, graph theory, group theory, stochastic processes, mathematical physics, among many others. 

Since their introduction evolution algebras have attracted the attention to several researchers, who were eager to investigate them from an algebraic point of view; see \cite{BCS}-\cite{Imo} and references therein. 

Here, we focus our attention on the evolution algebras whose square has dimension one; we classify them using the theory of inner product spaces and quadratic forms. We begin by introducing a few key ideas: any commutative algebra $A$ (not necessarily finite dimensional) over a field $K$ with $\dim(A^2)=1$ admits an inner product $\esc{\cdot,\cdot}$ such that the product in $A$ is given by $xy=\esc{x,y}a$, for some fixed element $a\in A$ (unique up to scalar multiples). From here, we obtain three (excluding) possibilities for $a$:
\begin{enumerate}
\item\label{one}
$a\in\Ann(A)$, which gives $A^3 = (A^2)^2 = 0$.
\item $a\notin \Ann(A)$ and $\esc{a,a} = 0$, which implies $(A^2)^2 = 0$ but $A^3\ne 0$.
\item $a\notin\Ann(A)$ and $\esc{a,a}\ne 0$, which yields $A^3 \ne 0$ and $(A^2)^2\ne 0$.
\end{enumerate}
Here $\Ann(A) = \{x \in A \, |\, xA = 0\}$. Thus, the algebras we will be dealing with come in three different flavours given by the trichotomy above. In any case, choosing a suplementary subspace $W$ of $\Ann(A)$ we obtain a decomposition $A = \Ann(A)\oplus W$ such that $\esc{\cdot, \cdot}\vert_W$ is nondegenerate. It is worth mentioning that we have quite a lot of flexibility to choose $W$, so depending on the flavour of our algebra we will require $W$ to satisfy certain conditions.

If $A$ is an evolution algebra of flavour \eqref{one}, we will show (in Theorem \ref{Iso1}) that the isomorphic class of $A$ is completely determined by the pair
$(\dim(A),[W])$, where $[W]$ denotes the isometry type of $W$. More precisely, if $A = \Ann(A)\oplus W_A$ and $B = \Ann(B) \oplus W_B$ are like in \eqref{one}, then we will prove that $A \cong B$ if and only if $\dim(A) = \dim(B)$, and $W_A$ and $W_B$ are isometric.

Similar results for flavours (2) and (3) will be explored. To do so, we will make used of the theory of inner products and/or quadratic forms.

The paper is organized as follows: in Section 2, we introduce the required background. Section 3 begins by noticing that our study of the evolution algebras $A$ with $\dim (A^2) = 1$ must be divided into two cases, attending on whether $(A^2)^2 \neq 0$ or $(A^2)^2 = 0$, which are treated in \S 3.1 and \S 3.2, respectively. 


\section{Preliminaries}

Throughout the paper, $V$ denotes a vector space over a field $K$. An {\bf inner product space} is a pair $(V, \esc{\cdot, \cdot})$, where $\esc{\cdot, \cdot}: V \times V \to K$ is a symmetric bilinear form. Two inner product spaces $(V,\esc{\cdot, \cdot})$ and $(V',\esc{\cdot, \cdot}')$ are said to be {\bf isometric} or {\bf equivalent} if there is a (vector space) isomorphism $f: V \to V'$ such that $\esc{f(x),f(y)}'= \esc{x,y}$ for all $x, y \in V$.  

\smallskip 

A map $q: V\to K$ satisfying that 
\begin{enumerate}
\item[\rm (1)] $q(\l v) = \l^2 q(v), \quad \forall \, \, \l \in K, \, \, v \in V$,
\item[\rm (2)] the map $\langle \cdot, \cdot \rangle_q: V \times V \to K$ given by $(x, y) \mapsto q(x + y) - q(x) - q(y)$, and called the {\bf polar form} of $q$, is bilinear,
\end{enumerate}
is called a {\bf quadratic form} on $V$, and $(V, q)$ is said to be a {\bf quadratic space}.


If the characteristic of $K$ is different from 2, the notions of inner product spaces and quadratic spaces are equivalent: the polar form of a quadratic form $q$ is now defined by 
$\langle x, y \rangle_q = \frac{1}{2}\left(q(x + y) - q(x) - q(y)\right)$, and satisfies that $\langle x, x\rangle_q = q(x)$. And conversely, if $(V,\esc{\cdot, \cdot})$ is an inner product space, then $q(x):= \esc{x, x}$ for all $x \in V$, is a quadratic form, whose polar form is precisely $\esc{\cdot, \cdot}$. In such a case, we may write $(V, q)$ to refer to the quadratic space $(V, \langle \cdot, \cdot \rangle)$, and vice versa.

\smallskip 

The {\bf radical} of an inner product space $(V, \langle \cdot, \cdot \rangle$) is the subspace of $V$ given by  
\[
V^\bot = \{x \in V  \, |\, \esc{x,V} = 0\},
\]
and $(V, \esc{\cdot, \cdot})$ is said to be {\bf nondegenerate} if $V^\bot = 0$; a {\bf subspace} $W$ of $V$ is called {\bf nondegenerate} if $(W, \esc{\cdot, \cdot}|_W)$ is nondegenerate. 
Recall that $V$ can be written as
\begin{equation} \label{decomV}
V = V^\bot \oplus W, 
\end{equation}
for $W$ a nondegenerate subspace of $V$.

If $V$ has finite dimension, then the matrix $M_B$ of $\langle \cdot, \cdot \rangle$ with respect to a basis $B$ of $V$ is called the {\bf Gram matrix} of $\esc{\cdot, \cdot}$ with respect to $B$. If $M_B$ and $M_{B'}$ are Gram matrices with respect to bases $B$ and $B'$ of $V$, then $M_B$ and $M_{B'}$ have the same rank; the {\bf rank} of $(V, \esc{\cdot, \cdot})$ is defined as the rank of a Gram matrix $M$ of $\esc{\cdot, \cdot}$, and $(V, \esc{\cdot, \cdot})$ is nondegenerate if and only if $M$ is nonsingular. The {\bf discriminant} of $(V, \esc{\cdot, \cdot})$ is defined as zero if $(V, \esc{\cdot, \cdot})$ is degenerate, and as the coset of $\det(M)$ in the factor group $K^*/{(K^*)}^2$, otherwise. The discriminant of two nondegenerate equivalent inner product spaces coincide. 
 
If the characteristic of $K$ is not 2, then we can consider the associated quadratic form $q$ of $(V, \esc{\cdot, \cdot})$ and define its rank. Similarly, one can define the discriminant of $q$ provided that $(V, \esc{\cdot, \cdot})$ is nondegenerate.

If $V$ has finite dimension $n$, then a real quadratic form $q: V \to \R$ of rank $r$ can be expressed as 
\[
q(x_1, \ldots, x_n) = x_1^2 + \ldots + x_p^2- x_{p+1}^2 - \ldots - x_r^2, 
\]
with respect to a suitable basis of $V$; the {\bf signature} of $q$ is defined as $(p, \, r - p)$.

Recall that (finite dimensional) inner product spaces over algebraically closed fields are classified (up to congruence) by rank; over the reals they are classified according to their rank and signature (see \cite[Theorem 6.8]{BAI}); while over finite fields of odd characteristic
their rank and discriminant constitute a complete set of invariants (see \cite[Theorem 6.9]{BAI}). Over other types of fields many different invariants are available; for instance, over a local field those are the dimension, discriminant and the so-called Hasse invariant (see \cite[p. 39]{Serre}). 
Lastly, diagonalizable inner product spaces over quadratically closed fields are classified by their rank.


\section{The first dichotomy}

An {\bf algebra} $A$ over $K$ is a vector space over $K$ endowed with a bilinear map $A \times A \to A$ written as $(a, b) \mapsto ab$, and called the {\bf product} of $A$. We say that $A$ is an {\bf evolution algebra} if there exists a basis $\{a_i\}_{i \in I}$, called a {\bf natural basis} of the underlying vector space of $A$ such that $a_ia_j = 0$ for all $i \neq j$. 

\begin{remark} \label{product}
Let $A$ be an algebra such that $\dim (A^2) = 1$. Then we can find $0 \neq a \in A$ such that $A^2 = Ka$, and the product of any two elements $x, y \in A$ is given by $xy = \lambda_{xy} a$, where $\lambda_{xy} \in K$ depends linearly on both $x$ and $y$. In other words, the map $\langle \cdot, \cdot \rangle: A \times A \to K$ given by $\langle x, y \rangle = \lambda_{xy}$, for all $x, y \in A$, is an inner product in $A$. Clearly, 
\begin{equation} \label{productA}
xy = \langle x, y \rangle a, \mbox{ for all } \, \, x, y \in A.
\end{equation} 
In this case, notice that \eqref{decomV} becomes 
\begin{equation} \label{decomA}
A = \mathrm{Ann}(A) \oplus W;
\end{equation}	
in other words $\Ann(A) = A^\bot$. 
\end{remark} 

At this point, it is worth mentioning that if $A$ is an evolution algebra, then \eqref{productA} does not depend of the generator $a$ of $A^2$. We prove this fact in a more general way using pointed vector spaces.

\begin{lemma} \label{pointed}
Let $(S, s)$ be a pointed vector space, where $0 \neq s \in S$, and $(U, \langle \cdot, \cdot \rangle)$ a nondegenerate inner product space. Then the direct sum $A_s := S \oplus U$ becomes an algebra with 1-dimensional square under the product
\[
(s_1 + u_1) (s_2 + u_2) = \langle u_1, u_2 \rangle s,
\]
for all $s_1, s_2 \in S$ and $u_1, u_2 \in U$. Moreover, if $s'$ is another nonzero element of $S$, 
then $A_s$ and $A_{s'}$ are isomorphic.
\end{lemma}

\begin{proof}
It is straightforward to check that $A_s$ is an algebra such that $\dim(A^2_s) = 1$. For $(S, s')$  another pointed vector space, take $\theta: S \to S$ a bijective linear map such that $\theta(s) = s'$. The map
$f:A_s \to A_{s'}$ given by 
$f(s + u) = \theta(s) + u$, for all $s \in S$ and $u \in U$, is the desired isomorphism.
\end{proof} 

\begin{proposition} \label{dicho}
Let $A$ be a commutative algebra with $\dim(A^2) = 1$. Then either $(A^2)^2 = 0$ or there is a unique nonzero idempotent in $A$.
\end{proposition}

\begin{proof}
From $\dim(A^2) = 1$ we can find $0 \neq a \in A$ such that $A^2 = Ka$. If $(A^2)^2 \neq 0$, then $0 \neq a^2 = \lambda a$, for some $0 \neq \lambda \in K$. We claim that $e := \lambda^{-1}a$ is the unique nonzero idempotent of $A$; in fact:
\[
e^2 = \lambda^{-2}a^2 = \lambda^{-2}(\lambda a) = \lambda^{-1}a = e.
\] 
Clearly, $A^2 = Ke$. If $0 \neq u \in A$ is an idempotent, then $u = u^2 \in A^2 = Ke$, and so $u = \gamma e$ for some $0 \neq \gamma \in K$. But then  
\[
\gamma e = u = u^2 = \gamma^2 e, 
\]
which implies that $\gamma = 1$, and so $u = e$.
\end{proof} 

\subsection{Case $(A^2)^2 \neq 0$} We first study the evolution algebras $A$ whose square is 1-dimensional and satisfy the condition $(A^2)^2 \neq 0$. 

\smallskip 

The following result follows from Remark \ref{product} and Proposition \ref{dicho}.

\begin{proposition} \label{inner}
Let $A$ be a commutative algebra such that $\dim(A^2) = 1$ and $(A^2)^2 \neq 0$. Then there exists a unique inner product $\langle \cdot, \cdot \rangle: A \times A \to K$ such that the product of $A$ is given by 
\[
xy = \langle x, y \rangle e, \, \, \forall \, \, x, y \in A,
\]
where $e$ is the nonzero unique idempotent of $A$. Moreover, $\langle e, e \rangle = 1$. 
\end{proposition}

\begin{definition}
The inner product defined in Proposition \ref{inner} is called the {\bf canonical inner product} of $A$, and $(A, \langle \cdot, \cdot \rangle)$ the {\bf canonical inner product space}. 
\end{definition}


\begin{proposition} 
Let $A$ be an evolution algebra such that $\dim(A^2) = 1$ and $(A^2)^2 \neq 0$. Then the canonical inner product of $A$ is diagonalizable, that is, there exists an orthogonal basis of $A$.
\end{proposition}

\begin{proof} 
Suppose that $\{a_i\}_{i\in I}$ is a natural basis of $A$. Then, for all $i \neq j$ we have 
\[
0 = a_ia_j = \langle a_i, a_j \rangle e,
\] 
which implies $\langle a_i, a_j \rangle = 0$. 
\end{proof}

\begin{conclu} \label{conclu1}
{\rm
We have proved that the product of an evolution algebra $A$ satisfying $(A^2)^2\ne  0$ and $\dim(A^2) = 1$,
is completely determined by a (unique) inner product $\langle \cdot, \cdot \rangle: A \times A \to K$ (diagonalizable with respect to a natural basis of $A$) and a (unique) nonzero idempotent $e$ of $A$ such that $\langle e, e \rangle = 1$. To be more precise, $xy = \langle x, y \rangle e$, for all $x, y \in A$. A kind of converse holds: if $(V, \langle \cdot, \cdot \rangle)$ is an inner product space with an orthogonal basis and a vector $v$ of norm one, then we can endow $V$ with an evolution algebra structure with the product 
$xy = \langle x, y \rangle v$, for all $x, y \in V$. Moreover, $\dim(V^2) = 1$.}
\end{conclu}

We can reformulate our problem using categories: 

\begin{itemize}
\item $\A$ denotes the category whose objects are the evolution $K$-algebras $A$ satisfying that $(A^2)^2 \ne 0$ and $\dim(A^2) = 1$, and morphisms the algebra homomorphisms; notice that $\A$ is a full subcategory of the category of all $K$-algebras;  
\item $\B$ denotes the category whose objects are triples of the form $\big(V, \langle \cdot, \cdot \rangle, v \big)$, where $(V, \langle \cdot, \cdot \rangle)$ is a (diagonalizable) inner product space, and $v \in V$ a norm one vector; a morphism $f: \big(V, \langle \cdot, \cdot \rangle, v) \to \big(V', \langle \cdot, \cdot \rangle', v'\big)$ in $\B$ is a linear map $f: V \to V'$ such that $\langle x, y \rangle f(v) = \langle f(x), f(y)\rangle'v'$, for all $x, y\in V$;
\item $\A^0$ stands for the full subcategory of $\A$ consisting of
finite dimensional evolution algebras, while $\B^0$ denotes the full subcategory of $\B$ of triples
$\big(V, \langle \cdot, \cdot \rangle, v \big)$, where $V$ is finite dimensional.
\end{itemize}

We begin by characterizing the isomorphisms in $\B$: 

\begin{theorem} \label{caractIso}
A morphism $f: \big(V, \langle \cdot, \cdot \rangle, v \big) \to \big(V', \langle \cdot, \cdot \rangle', v'\big)$ in $\B$ is an isomorphism if and only if $f$ is an isometry and $f(v) = v'$.
\end{theorem}

\begin{proof} 
Suppose that $f$ is an isomorphism. Then for $\lambda := \langle f(v), f(v) \rangle' \in K$, we have that 
\[
f(v) = \langle f(v), f(v)\rangle' v' = \lambda v',
\]
since $f$ is a morphism in $\B$ and $\langle v, v \rangle = 1$. From here we obtain that
\[
\lambda v' = \langle \lambda v', \lambda v' \rangle' v' = \lambda^2 v',
\]
which implies that $\lambda = 1$. Thus $f(v) = v'$, and $f$ is an isometry. The converse clearly holds.
\end{proof}

\begin{theorem}
The categories $\A$ and $\B$ are isomorphic. 
\end{theorem}

\begin{proof}
The functors $F: \A \to \B$ and $G: \B \to \A$ mapping an evolution algebra $A$ to the triple $\big(A, \langle \cdot, \cdot \rangle, e\big)$ (where $e$ is the unique nonzero idempotent of $A$, and $\langle \cdot, \cdot \rangle$ is its canonical inner product), and a triple $\big(V, \langle \cdot, \cdot \rangle, v\big)$ to the evolution algebra $V$ whose product is $xy := \langle x, y \rangle v$, for all $x, y \in V$, respectively, are well-defined by Proposition \ref{inner} and Conclusion \ref{conclu1}. It is straightforward to check that  
$FG = 1_{\B}$ and $GF = 1_{\A}$. 
\end{proof}

A very important consequence of this theorem is the following: 

\begin{corollary}
The problem of classifying (up to isomorphism) the evolution algebras in $\A$ is equivalent to classifying (up to isomorphism) the triples $\big(V, \langle \cdot, \cdot \rangle, v\big)$ in $\B$.  
\end{corollary}

Keeping in mind that two objects in $\B$ are isomorphic in the presence of an isometry, what we are indeed doing here is transitioning from an algebraic classification problem to a geometric classification problem.


\begin{remark} \label{new}
If $K$ is quadratically closed and $(V, \esc{\cdot, \cdot})$ is finite dimensional, diagonalizable and nondegenerate, then there exists a basis $B$ of $V$ such that the Gram matrix $M_B$ is the identity. As a consequence, any two diagonalizable, nondegenerate inner product spaces of the same dimension are isometric. Moreover, if $(V_1,\esc{\cdot, \cdot}_1)$ and $(V_2,\esc{\cdot, \cdot}_2)$ are such that $\dim(V_1) = \dim(V_2)$ and $\dim(V^\bot_1) = \dim(V^\bot_2)$, then $(V_1,\esc{\cdot, \cdot}_1)$ and $(V_2,\esc{\cdot, \cdot}_2)$ are isometric.
\end{remark}

Using Witt's (Isometry) Extension Theorem, we can provide a way to construct isomorphic objects in $\B^0$. Before doing so, we remind the reader a trivial property of vector spaces very useful for our purposes.

\begin{remark} \label{complemento}
Let $V$ be a vector space and $U$ a subspace of $V$. If $v \in V$ is such that $v \notin U$, then there exists a subspace $U'$ of $V$ such that $v \in U'$ and $V = U \oplus U'$. 
\end{remark}

\begin{lemma}\label{admunsen}
Let $K$ be a field of characteristic different from two or perfect of characteristic two and $(V, q, v) \in \B^0$. Suppose that $V = V^\bot \oplus V'$ for $V'$ a subspace of $V$ containing $v$. If $w \in V'$ satisfies that $q(w) = 1$, then $(V, q, v)$ and $(V, q, w)$ are isomorphic in $\B^0$. 
\end{lemma}

\begin{proof}
Suppose first that the characteristic of $K$ is not two. 
Then the linear map from $Kv'$ onto $Kw'$ mapping $v'$ onto $w'$ is an isometry, which can be extended to an isometry $\theta'$ of $V'$, by Witt's (Isometry) Extension Theorem. It is straightforward to check that the map $\theta: V \to V$ given by $\theta(t + x) =  t + \theta'(x)$, for all $t \in V^\bot$ and $x \in V'$, is an isometry of $V$ mapping $v$ to $w$; the result follows from Theorem \ref{caractIso}.

Assume now that $K$ is perfect of characteristic two. In this case, $V$ can be written as orthogonal direct sums $V = Kv \oplus V' = Kw \oplus W'$ such that $V^\bot = V'^\bot = W'^\bot$. 
Thus, $V' = V^\bot \oplus V''$ and $W'= V^\bot \oplus W''$, which imply that $V = Kv \oplus V^\bot \oplus V'' = Kw \oplus V^\bot \oplus W''$. Thus $(V'', \esc{\cdot, \cdot}|_{V''})$ and $(W'', \esc{\cdot, \cdot}|_{W''})$ are nondegenerate and have the same dimension, and so they are isometric by Remark \ref{new}. If $\theta'': V'' \to W''$ is an isometry, then we can easily construct an isometry $\theta: V \to W$ such that $\theta|_{V''} = \theta''$, $\theta(v) = w$.
\end{proof}

An immediate consequence of Lemma \ref{admunsen} in terms of isomorphisms of evolution algebras follows: 

\begin{theorem} \label{uncaso} 
Let $K$ be a field of characteristic different from two or perfect of characteristic two.
Two evolution algebras $A$ and $B$ in $\A$ are isomorphic if and only if their canonical inner product spaces are isometric. Moreover, if $K$ is algebraically closed, or of characteristic two and perfect, then $A$ and $B$ are isomorphic if and only if the rank of their canonical inner products coincide; if $K = \R$, then $A$ is isomorphic to $B$ if and only if the rank and signatures of their canonical inner products coincide.
\end{theorem}

We close this first case with some concrete examples: 

\begin{example}
In dimension 4, Theorem \ref{uncaso} tells us that there are four different isomorphic classes in $\mathcal{A}_\C$, which correspond 
to the triples $(\C^4, q_i, e_1)$, where $e_1 = (1, 0, 0, 0)$, and $q_1, q_2, q_3, q_4$ are as follows:
\begin{align*}
q_1(x_1, x_2, x_3, x_4) & = x_1^2,
\\
q_2(x_1, x_2, x_3, x_4) & = x_1^2 + x_2^2,
\\
q_3(x_1, x_2, x_3, x_4) & = x_1^2 + x_2^2 + x_3^2,
\\
q_4(x_1, x_2, x_3, x_4) & = x_1^2 + x_2^2 + x_3^2 + x_4^2,
\end{align*}
with respect to the canonical basis of $\C^4$.
\end{example}

\begin{example}
In dimension 3, Theorem \ref{uncaso} reveals that we have six different isomorphic classes in 
$\mathcal{A}_\R$ corresponding to the triples $(\R^3, q_i, e_1)$, where $e_1 = (1, 0, 0)$ and $q_1, \ldots, q_6$ are displayed in the table below, where the coordinates are with respect to the canonical basis of $\R^3$.
\begin{center} 
\begin{tabular}{|c|c|c|}
\hline
$q_i$ & rank & signature\cr
\hline 
$x_1^2$                 & $1$ & $(1,0)$  \cr 
$x_1^2 + x_2^2$         & $2$ & $(2,0)$  \cr 
$x_1^2 - x_2^2$         & $2$ & $(1,1)$  \cr 
$x_1^2 + x_2^2 + x_3^2$ & $3$ & $(3,0)$  \cr 
$x_1^2 + x_2^2 - x_3^2$ & $3$ & $(2,1)$  \cr 
$x_1^2 - x_2^2 - x_3^2$ & $3$ & $(1,2)$  \cr 
\hline
\end{tabular}
\end{center}
\end{example}

\begin{example}
Let $\mathbf{F}_4$ be the field of four elements, which is perfect. In dimension 3, Theorem \ref{uncaso} tells us that there are three distinct classes in $\mathcal{A}_{\mathbf{F}_4}$, which correspond to the inner products (on the $\mathbf{F}_4$-vector space $\mathbf{F}^3_4$) whose matrices are $\mathrm{Id}$, and the diagonal matrices 
$\mathrm{diag}(1, 1, 0)$ and $\mathrm{diag}(1, 0, 0)$.
\end{example} 

\begin{remark}
Recall that quadratic forms on a finite dimensional vector space over a finite field of odd characteristic are classified (up to congruence) by their rank and discriminant. Thus $(V_1,q_1,v_1) \cong (V_2, q_2, v_2)$ in $\B^0$ if and only if 
$\dim(V_1) = \dim(V_2)$, $\dim(V_1^\bot) = \dim(V_2^\bot)$ and the discriminant of ${q_1}\vert_{V'_1}$ coincides with that of ${q_2}\vert_{V'_2}$, where $v_i\in V'_i$ and
$V'_i$ is the complement of $V_i^\bot$, for $i = 1, 2$.
\end{remark}

\begin{example}
Let $K = \mathbf{F}_3(i) = \{0, \pm 1, \pm i, \pm(1 + i), \pm(1 - i)\}$, where $i^2 = -1$, be the field of nine elements. Notice that $K$ can be seen as an extension of the field of three elements $\mathbf{F}_3$ by adjoining an element of square $-1$.
We have that $(K^\times)^2 = \{ \pm 1, \pm i\}$ is the cyclic group of order $4$, and the quotient group 
$K^\times/(K^\times)^2$ is the cyclic group of order $2$. We can then write $K^\times/(K^\times)^2 = \{[1],[\omega]\}$, where $\omega = 1 + i$ and $[\cdot]$ denotes the corresponding equivalence class.
From here we obtain that the discriminant of a quadratic form over $K$ is either $[1]$ or $[\omega]$.

In particular, for $V = K^n$ and $q: V \to K$ nondegenerate, we have two possibilities: either $q$ is congruent to $x_1^2 +x_2^2+ \cdots + x_n^2$ or to $\omega x_1^2 + x_2^2 + \cdots + x_n^2$. A natural question arises: how many isomorphic classes of $3$-dimensional evolution $K$-algebras $A$ with $\dim(A^2)=1$ and $(A^2)^2\ne 0$ are there?
\newline 
We obtain the three following types: 
\begin{enumerate}
\item $\hbox{Ann}(A)=0$, then $A\cong K^3$ with product
\begin{align*}
(x,y,z)(x',y',z') & = (xx'+yy'+zz')(1,0,0), \mbox{ or }
\\
(x,y,z)(x',y',z') & = (\omega xx'+yy'+zz')(1,0,0).
\end{align*}
\item $\dim(\hbox{Ann}(A))=1$, then $A\cong K\times K^2$ with product 
\begin{align*}
(x,y,z)(x',y',z') &= (yy'+zz')(0,1,0), \mbox{ or }
\\ 
(x,y,z)(x',y',z') &= (\omega yy'+zz')(0,1,0).
\end{align*}
\item $\dim(\hbox{Ann}(A))=2$, then $A\cong K^2\times K$ with  product 
\begin{align*}
(x,y,z)(x',y',z') & =zz'(0,0,1), \mbox{ or } 
\\
(x,y,z)(x',y',z') &= \omega zz' (0,0,1).
\end{align*}
\end{enumerate}
\end{example}

\subsection{Case $(A^2)^2 = 0$} We treat now the remaining case: evolution algebras $A$ such that $\dim (A^2) = 1$ and $(A^2)^2 = 0$. 

\smallskip

\begin{lemma} \label{lemmaanndecomp}
Let $A$ be an evolution $K$-algebra such that $\dim(A^2) = 1$ and $(A^2)^2 = 0$. Then 
$A = \Ann(A) \oplus W$, where
$(W, \esc{\cdot, \cdot}\vert_W)$ is nondegenerate and has an orthogonal basis $\{w_i\}_{i \in I}$ of nonisotropic vectors. 
Moreover, if $K$ is quadratically closed, then $\{w_i\}_{i \in I}$ is an orthonormal basis. 
\end{lemma}

\begin{proof}
By \eqref{decomA} we have $A = \mathrm{Ann}(A) \oplus W$, for $W$ a subspace of $A$ such that $(W, \esc{\cdot, \cdot}\vert_W)$ is nondegenerate. Take $\{e_i\}_{i \in J}$ a natural basis of $A$, and express each $e_i$ as $e_i = r_i + w_i$, for $r_i \in \Ann(A)$ and $w_i \in W$. 
Let $J_W = \{i \in J \mid w_i \neq 0\}$. We claim that $\{w_i\}_{i \in J_W}$ is a basis of $W$; in fact, it clearly spans $W$, and for $i \neq j$ we have that
\[
0 = e_ie_j = (r_i + w_i)(r_j + w_j) = r_ir_j + r_iw_j + w_ir_j + w_iw_j = w_iw_j,
\] 
since $r_i, r_j \in \Ann(A)$. This shows that the $w_i's$ are pairwise orthogonal. Now if $w_i \neq 0$ and $w^2_i = 0$ for some $i$, then $\langle w_i, w_i \rangle = 0$, and so $\langle w_i, W \rangle = 0$, which implies that $w_i \in W^\bot \cap W = 0$, a contradiction. Thus $\langle w_i, w_j \rangle = 0$ if $i \neq j$, and $\langle w_i, w_i \rangle \neq 0$ (provided that $w_i \neq 0$), which implies that  
the $w_i's$ are linearly independent, and so $\{w_i\}_{i \in J_W}$ is a basis of $W$.

To finish, notice that if $K$ is quadratically closed, then the set $\Big\{\frac1{\sqrt{\langle w_i, w_i\rangle}} w_i \Big\}_{i \in J_W}$ is an orthonormal basis of $W$. This finishes the proof.
\end{proof} 

\smallskip 

Let $A$ be an evolution algebra such that $\dim (A^2) = 1$ and $(A^2)^2 = 0$. We proceed like in Remark \ref{product}, and we write the product in $A$ as 
\[
xy = \langle x, y \rangle a, \, \mbox{ for all } \, x, y \in A,
\]
where $a \in A$ satisfies $A^2 = Ka$.
Notice that $(A^2)^2 = 0$ implies $a^2 = 0$, and so $\langle a, a \rangle = 0$, which says that $a$ is isotropic. In what follows, we distinguish two sub cases depending on whether $a \in \Ann(A)$ or $a \notin \Ann(A)$.

\subsubsection{{\bf Sub case:} $a \in \Ann(A)$, or equivalently, $A^3 = 0$}


It turns out that $A$ is associative (and commutative) in this case.

\begin{remark}
Let $A$ be an evolution algebra as in Lemma \ref{lemmaanndecomp}. If $A^3 = 0$, then the basis $\{w_i\}$ of $W$ is such that $w^2_i \in \Ann(A)$ for all $i$.
\end{remark}

\begin{theorem} \label{Iso1}
Let $A$ and $B$ be evolution algebras satisfying that $\dim(A^2) = \dim(B^2) = 1$, $(A^2)^2 = A^3 = 0$ and $(B^2)^2 = B^3 = 0$. 
Write $A = \Ann(A) \oplus W_A$ and $B = \Ann(B) \oplus W_B$ as in \eqref{decomA}.
Then $A$ and $B$ are isomorphic if and only if $\dim(\Ann(A)) = \dim(\Ann(B))$, and the spaces $W_A$ and $W_B$ are isometric.
\end{theorem}

\begin{proof}
Suppose first that $\theta: (W_A, \langle \cdot, \cdot \rangle \vert_{W_A}) \to (W_B, \langle \cdot, \cdot \rangle \vert_{W_B})$ is an isometry 
and that $\dim(\Ann(A)) = \dim(\Ann(B))$. For a linear isomorphism $\xi: \Ann(A) \to \Ann(B)$, one can easily check that the map $F: A \to B$ given by $F(y + z) = \xi(y) + \theta(z)$, for all $y \in \Ann(A)$ and $z \in W_A$, is the desired isomorphism. 

Conversely, suppose that $F: A \to B$ is an isomorphism. Then $B = \hbox{Ann}(B)\oplus F(W_A)$, which implies $W_B = F(W_A)$. If $a \in A$ is such that $A^2 = Ka$, then Lemma \ref{pointed} allows us to choose $b = F(a)$, as the generator of $B^2$.
It is clear that $F\vert_{\Ann(A)}$ is a linear isomorphism, and so $\dim(\Ann(A)) = \dim(\Ann(B))$. It remains to show that $\theta = F\vert_{W_A}$
is an isometry; in fact, for $z_1, z_2 \in W_A$ we have that 
\[
z_1z_2  = \langle z_1, z_2 \rangle a,
\]
which implies that 
\[
\langle z_1, z_2 \rangle b = F(z_1z_2) = F(z_1)F(z_2) = \theta(z_1)\theta(z_2) = 
\langle \theta(z_1),  \theta (z_2) \rangle b.
\]
Thus: $\langle z_1, z_2 \rangle = \langle \theta(z_1),  \theta (z_2) \rangle$, proving that $\theta$ is an isometry, as desired. 
\end{proof}

\begin{example}
Using Theorem \ref{Iso1} we can determine the 3-dimensional real evolution algebras $A$ such that $\dim(A^2) = 1$ and $(A^2)^2 = A^3 = 0$. In fact, for 
$d = \dim(\Ann(A))$, we have the following cases:
\begin{itemize}
\item $d = 2$: We have that $A \cong \R^3$ with $\dim(W) = 1$. There are two nonisomorphic algebras with products given by 
\[
(x_1, x_2, x_3)(y_1, y_2, y_3) = (\pm x_3y_3, 0, 0).
\]
\item $d = 1$: We have that $A \cong \R^3$ with $\dim(W) = 2$. There are three nonisomorphic algebras with products given by  
\begin{align*}
(x_1, x_2, x_3)(y_1, y_2, y_3) &= (x_2y_2 + x_3y_3, 0, 0);
\\
(x_1, x_2, x_3)(y_1, y_2, y_3) &= (x_2y_2 - x_3y_3, 0, 0);
\\
(x_1, x_2, x_3)(y_1, y_2, y_3) &= (-x_2y_2 - x_3y_3, 0, 0).
\end{align*}
\item $d = 0$: We have that $A \cong \R^3$ with $\dim(W) = 3$. There are four nonisomorphic algebras with products given by  
\begin{align*}
(x_1, x_2, x_3)(y_1, y_2, y_3) &= (x_1y_1 + x_2y_2 + x_3y_3, 0, 0);
\\
(x_1, x_2, x_3)(y_1, y_2, y_3) &= (x_1y_1 + x_2y_2 - x_3y_3, 0, 0);
\\
(x_1, x_2, x_3)(y_1, y_2, y_3) &= (x_1y_1 -x_2y_2 - x_3y_3, 0, 0);
\\
(x_1, x_2, x_3)(y_1, y_2, y_3) &= (-x_1y_1 -x_2y_2 - x_3y_3, 0, 0).
\end{align*}
\end{itemize}
\end{example}

We can improve Theorem \ref{Iso1}, provided that the field $K$ is quadratically closed. 

\begin{theorem} \label{A3=0}
Let $K$ be a quadratically closed field.
Then the isomorphic class of an evolution algebra $A$ with $\dim(A^2) = 1$ and $(A^2)^2 = A^3 = 0$ is completely determined by $\dim (A)$ and $\dim \big(\Ann(A)\big)$.
\end{theorem}

\begin{proof}
Let $B$ be an evolution algebra with $\dim(B^2) = 1$ and $(B^2)^2 = B^3 = 0$. If $\dim(B) = \dim(A)$ and $\dim \big(\Ann(B)\big) = \dim \big(\Ann(A)\big)$, then $\dim(W_A) = \dim(W_B)$ by \eqref{decomA}. Thus, $W_A$ and $W_B$ are isometric by Lemma \ref{lemmaanndecomp}.

\noindent The converse follows from Theorem \ref{Iso1} and \eqref{decomA}.
%
%
%
\end{proof}

\begin{example}
Over $\C$, in dimension 4, Theorem \ref{A3=0} tells us that there are four different (isomorphic) classes of evolution algebras $A$ with $\dim(A^2) = 1$, $(A^2)^2 = A^3 = 0$ for $\dim(\Ann(A)) = 1$. Their quadratic forms are given by
\begin{align*}
q_1(x_1, x_2, x_3, x_4) & = x_2^2 + x_3^2 + x_4^2,
\\
q_2(x_1, x_2, x_3, x_4) & = x_3^2 + x_4^2,
\\
q_3(x_1, x_2, x_3, x_4) & = x_4^2,
\end{align*}
with respect to the canonical basis of $\C^4$.
\end{example}

\begin{example}
Let us now classify the $3$-dimensional evolution algebras $A$ with $\dim(A^2) = 1$ and $(A^2)^2 = A^3 = 0$ over the field $K$ of nine elements.
In this case the existence of a nonzero annihilator is compulsory, and we obtain two types: 
\begin{enumerate}
\item $\dim(\hbox{Ann}(A))=1$, then $A\cong K\times K^2$ with  product 
\begin{align*}
(x,y,z)(x',y',z') & =(yy'+zz')(1,0,0), \mbox{ or }
\\
(x,y,z)(x',y',z') & =(\omega yy'+zz')(1,0,0).
\end{align*}
\item $\dim(\hbox{Ann}(A))=2$, then $A\cong K^2\times K$ with product
\begin{align*}(x,y,z)(x',y',z') &= zz'(1,0,0), \mbox{ or }
\\ 
(x,y,z)(x',y',z') &= \omega zz' (1,0,0).
\end{align*}
 \end{enumerate}
 \end{example}
 
\subsubsection{{\bf Sub case:} $A^3 \neq 0$} Suppose that $A = \Ann(A) \oplus W$ is as in \eqref{decomA}. If $a = x + w$, for $x \in \Ann(A)$ and $w \in W$, then $w$ is also isotropic; in fact:
\begin{equation} \label{isotropico}
0 = \langle a, a \rangle = \langle x, x \rangle + 2 \langle x, w \rangle + \langle w, w \rangle = \langle w, w \rangle.
\end{equation}

The proof of the next result is very similar to the proof of Theorem \ref{Iso1}. We provide a sketch of it and leave the details to the reader. 

\begin{theorem} \label{Iso2}
Let $A$ and $B$ be evolution algebras such that $\dim(A^2) = \dim(B^2) = 1$, $(A^2)^2 = 0$, $(B^2)^2 = 0$ and that both $A^3$ and $B^3$ are nonzero. Let $a \in A$ and $b \in B$ such that $A^2 = Ka$ and $B^2 = Kb$. Suppose that $A = \Ann(A) \oplus W_A$ and $B = \Ann(B) \oplus W_B$ as in \eqref{decomA}, and let $a = x + w$, $b = x' + w'$, where $x \in \Ann(A)$, $x' \in \Ann(B)$, $w \in W_A$ and $w' \in W_B$.
Then $A$ and $B$ are isomorphic if and only if $\dim(\Ann(A)) = \dim(\Ann(B))$, and there exists an isometry $\theta: W_A \to W_B$ such that $\theta(w) = w'$.
\end{theorem}

\begin{proof}
If $\dim(\Ann(A)) = \dim(\Ann(B))$ and $\theta: W_A \to W_B$ is an isometry such that $\theta(w) = w'$, then choose a linear isomorphism $\xi: \Ann(A) \to \Ann(B)$ such that $\xi(x) = x'$ and proceed like in the proof of Theorem \ref{Iso1}.

For the converse, if $F: A \to B$ is an isomorphism, reason like in the proof of Theorem \ref{Iso1} by noticing that $F(x) = x'$ and $F(w) = w'$.
\end{proof}

\begin{example} \label{F4ejemplo2}
Let $K$ be the field of $4$ elements: $K = \hbox{\bf F}_4 = \{0, 1, \alpha, \beta\}$, where $\alpha + \beta = 1$, $\alpha^2 = \beta$, $\beta^2 = \alpha$ and $\alpha\beta = 1$. Since $K$ is perfect, any orthogonalizable nondegenerate inner product admits an orthonormal basis. This implies that (up to isometry) there is only one nondegenerate orthogonalizable inner product. Therefore, we can 
consider the inner product $\esc{\cdot, \cdot}: K^2 \times K^2 \to K$ defined by $\esc{(x, y), (x', y')} = xx' + yy'$, for all $x, x', y, y' \in K$. Let us begin by computing the group $\mathcal{O}(K^2, \esc{\cdot,\cdot})$ of isometries of $(K^2, \esc{\cdot,\cdot})$. This can be easily done by identifying linear maps $K^2 \to K^2$ with their matrices with respect the canonical basis. In fact, a given linear map is an isometry if and only its matrix $M$ satisfies that $MM^t = M^2 = 1$. Proceeding in this way we obtain that
\[
\mathcal{O}(K^2, \esc{\cdot,\cdot}) = 
\left \{ 
\left(\begin{array}{@{}cc@{}}
1 & 0 
\\ 
0 & 1 
\end{array} \right), \, 
\left(\begin{array}{@{}cc@{}}
0 & 1 
\\ 
1 & 0 
\end{array} \right), \,
\left(\begin{array}{@{}cc@{}}
\alpha & \beta 
\\ 
\beta & \alpha
\end{array} \right), \, 
\left(\begin{array}{@{}cc@{}}
\beta & \alpha \\
\alpha & \beta
\end{array} \right)
\right \},
\]
which is isomorphic to the Klein group $\hbox{\bf F}_2 \times \hbox{\bf F}_2$. Next, notice that the nonzero isotropic vectors are 
\[
K^\times(1, 1) = \big \{(1, 1), (\alpha, \alpha), (\beta, \beta) \big \} \cong \hbox{\bf F}_3.
\]
Thus, there are three singletons orbits under the natural action of the group $\mathcal{O}(K^2, \esc{\cdot,\cdot})$ on the set $K^\times(1, 1)$, namely: $\{(1, 1)\}$, $\{(\alpha, \alpha)\}$ and $\{(\beta, \beta)\}$. 

We are now in a position to determine all the 3-dimensional evolution algebras $A$ such that 
$\dim(A^2) = 1$, $(A^2)^2 = 0$, $A^3 \neq 0$ and $\dim(\Ann(A)) = 1$. Let us take $A$ one of these evolution algebras, and write $A = \Ann(A) \oplus W$ by \eqref{decomA}. Then $\dim(W) = 2$, and we can identify $(W, \langle \cdot, \cdot \rangle)\vert_W$ with $K^2$ endowed with the inner product space having the identity matrix (with respect to the canonical basis). Lemma \ref{pointed} allows us to choose the generator of $A^2$ of the form $a = (1, \lambda, \mu)$, where $(\lambda, \mu) \in K^\times(1, 1)$. In total, there are three possibilities for $(\lambda, \mu)$, which induce non-isomorphic evolution algebras. Theorem \ref{Iso2} allows us to conclude that (up to isomorphism) there are three evolution algebras with products:
\begin{align*}
(x, y, z)(x', y',z') &= (yy' + zz')(1, 1, 1),
\\
(x, y, z)(x', y',z') &= (yy' + zz')(1, \alpha, \alpha),
\\
(x, y, z)(x', y', z') &= (yy' + zz')(1, \beta, \beta).
\end{align*} 
\end{example}

\begin{example} 
The $3$-dimensional evolution algebras $A$ with $\dim(A^2)=1$ such that $(A^2)^2=0$ but $A^3\ne 0$ over the field $K$ of nine elements are of two types: 
\begin{enumerate}
\item $\hbox{Ann}(A)=0$, then $A\cong K^3$ with product 
\begin{align*}
(x,y,z)(x',y',z') &=(xx'+yy'+zz')(0,1,i), \mbox{ or }
\\
(x,y,z)(x',y',z') &=(\omega xx'+yy'+zz')(0,1,i). 
\end{align*}
\item $\dim(\hbox{Ann}(A))=1$, then $A\cong K\times K^2$ with product 
\[
(x,y,z)(x',y',z')=(yy'+zz')(0,1,i).
\]
\end{enumerate}
Notice that in (2), inner products of the form 
\[
\langle (x, y, z), (x', y', z') \rangle = \omega yy' + zz',
\]
can not be considered since they do not have nonzero isotropic vectors. 
\end{example}


\begin{remark}
Let $K$ be a field of characteristic two, and $\mathcal{C}_{nd}$ the class of $n$-dimensional evolution algebras $A$ satisfying that $\dim(A^2) = 1$, $(A^2)^2 = 0$, $A^3 \neq 0$, and 
$d = \dim(\Ann(A))$. Notice that Example \ref{F4ejemplo2} shows us that the isomorphic clases of $\mathcal{C}_{nd}$ are in one to one correspondence with the orbits of the group $\mathcal{O}(K^{n-d}, \esc{\cdot,\cdot})$ (of isometries of $K^{n - d}$) on the set of isotropic vectors.  
\end{remark}

As expected, we can improve Theorem \ref{Iso2} by requiring the field to be quadratically closed. Before doing so, we need to prove a few results, the first one is a reformulation of the famous result known as {\it Witt's Cancelation Theorem}. 

\begin{proposition} \label{WCT2.0}
Let $V_1$ and $V_2$ be vectors spaces over the same field of characteristic not two, $q_i$ a nondegenerate quadratic form on $V_i$, and $U_i$ a nondegenerate subspace of $V_i$, for $i = 1, 2$. If the spaces $(V_1, q_1)$ and $(V_2, q_2)$ are isometric, and there is an isometry $U_1 \to U_2$, then  
there is an isometry $U_1^\bot \to U_2^\bot$.
\end{proposition}

\begin{proof}
Suppose that $f: V_1 \to V_2$ and $g: U_1 \to U_2$ are isometries. Then the map $gf^{-1}: f(U_1)\to U_2$ is also an isometry. Witt's Cancellation Theorem tells us that there is an isometry between 
$f(U_1)^\bot = f(U_1^\bot)$ and $U_2^\bot$, say $h$. To finish, notice that the composition 
$h f\vert_{U_1^\bot}$ is the desired isometry.
\end{proof}

\begin{lemma} \label{Isometria}
Let $(W_1, q_1)$ and $(W_2, q_2)$ be nondegenerate isometric spaces over a field $K$ of characteristic not two. If $w_1 \in W_1$ and $w_2 \in W_2$ are nonzero
isotropic vectors, then there exists an isometry $f: W_1 \to W_2$ such that $f(w_1) = w_2$.
\end{lemma}

\begin{proof}
The result trivially holds in dimension 1. Suppose now that both $W_1$ and $W_2$ have dimension $\geq 2$. Write $\esc{\cdot,\cdot}_i$ to denote the polar form of $q_i$, for $i = 1, 2$. There exists $w'_i \in W_i$ such that $(w_i, w'_i)$ is a hyperbolic pair in $W_i$, for $i = 1, 2$. It is straightforward to check that the linear map $\xi: Kw_1\oplus Kw_1'\to Kw_2\oplus Kw_2'$ such
that $\xi(w_1) = w_2$ and $\xi(w_1') = w_2'$ is an isometry. If $\hbox{dim}(W_i) = 2$, then we are done. Otherwise, Proposition \ref{WCT2.0} gives 
an isometry $\eta: (Kw_1\oplus Kw_1')^\bot \to (Kw_2\oplus Kw_2')^\bot$. Now, $V_i = (Kw_i \oplus Kw_i')\oplus (Kw_i\oplus Kw_i')^\bot$, for $i = 1, 2$, and we can easily construct an isometry $f: W_1 \to W_2$ such that $f(w_1) = w_2$, $f(w'_1) = w'_2$ and $f\vert_{(Kw_1\oplus Kw_1')^\bot} = \eta$, finishing the proof.
\end{proof}

\begin{theorem} \label{A3not0}
Let $K$ be a quadratically closed field of characteristic not two. 
The isomorphic class of an evolution algebra $A$ with $\dim(A^2) = 1$, $(A^2)^2 = 0$ and $A^3 \neq 0$ is completely determined by $\dim (A)$ and $\dim \big(\Ann(A)\big)$. 
\end{theorem}

\begin{proof}
Let $B$ be an evolution algebra such that $\dim(B^2) = 1$, $(B^2)^2 = 0$ and $B^3 \neq 0$. 
We express $A$ and $B$ as in \eqref{decomA}:
\begin{equation} \label{ABDec}
A = \Ann(A) \oplus W_A, \quad 
B = \Ann(B) \oplus W_B.
\end{equation}
We write $a = x + w$, $b = x' + w'$, where $x \in \Ann(A)$, $x' \in \Ann(B)$, and $w \in W_A$, $w' \in W_B$, and $a$ (respectively, $b$) generates $A^2$ (respectively, $B^2$). Notice that $w$ and $w'$ are both isotropic by \eqref{isotropico}.

Suppose first that $\dim(A) = \dim(B)$ and $\dim(\Ann(A)) = \dim(\Ann(B))$. Then $\dim(W_A) = \dim(W_B)$,   and Lemma \ref{lemmaanndecomp} applies to get that the nondegenerate spaces $W_A$ and $W_B$ are isometric. Lemma \ref{Isometria} and Theorem \ref{Iso2} tell us that $A$ and $B$ are isomorphic. 

The converse follows from \eqref{ABDec} and Theorem \ref{Iso2}.
%
\end{proof}

\end{document}